\documentclass[reqno, 12pt, twoside]{amsart}

\usepackage[margin=3cm]{geometry}

\numberwithin{equation}{section}

\usepackage[T1]{fontenc}
\usepackage[english]{babel}
\usepackage{amsthm}
\usepackage{amsfonts}
\usepackage{amsmath}
\usepackage{amssymb}

\usepackage{multicol}

\usepackage{mathtools}
\usepackage{enumerate}
\usepackage{hyperref}
\usepackage{array}

\usepackage{enumitem}

\DeclareMathOperator{\Res}{Res}

\newcommand{\Q}{\mathbb{Q}}

\newcommand{\PP}{\mathbb{P}}
\newcommand{\Z}{\mathbb{Z}}

\newcommand{\ZZ}{\mathbb{Z}}
\newcommand{\QQ}{\mathbb{Q}}
\newcommand{\ve}{\varepsilon}

\newcommand{\FF}{\mathbb{F}}

\newcommand{\NN}{\mathbb{N}}

\renewcommand{\epsilon}{\varepsilon}

\newtheorem{theorem}{Theorem}[section]
\newtheorem{corollary}[theorem]{Corollary}
\newtheorem{lemma}[theorem]{Lemma}

\newtheorem{conjecture}[theorem]{Conjecture}

\theoremstyle{definition}

\renewcommand{\leq}{\leqslant}
\renewcommand{\geq}{\geqslant}

\usepackage[usenames,dvipsnames]{color}

\begin{document}

\author{Tim Browning}
\author{Stephanie Chan}
\address{IST Austria\\
Am Campus 1\\
3400 Klosterneuburg\\
Austria}
\email{tdb@ist.ac.at}
\email{stephanie.chan@ist.ac.at}

\title{Solubility of a resultant equation and applications}

\date{\today}

\begin{abstract}
The large sieve is used to estimate the density of quadratic polynomials $Q\in \ZZ[x]$, such that there exists an odd degree polynomial defined over $\ZZ$ which 
has  resultant 
 $\pm 1$  with $Q$. 
 Given a monic polynomial $R\in \ZZ[x]$ of odd degree, this is  used to show that for almost all quadratic  polynomials $Q\in \ZZ[x]$, there exists a prime $p$ such that $Q$ and $R$ share  a common root in $\overline\FF_p$.
 Using recent work of Landesman, an application to the average size of the  odd part of the class group of  quadratic number fields is also given.
 \end{abstract}

\subjclass[2010]{11R29 (11G50, 11N36, 11R11, 11R45)}

\maketitle

\thispagestyle{empty}
 \setcounter{tocdepth}{1}
 {\small
 \begin{multicols}{2}
 \tableofcontents
 \end{multicols}
 }

\section{Introduction}

Let $Q\in \ZZ[x]$ be a quadratic polynomial
and let $R\in \ZZ[x]$ be a  polynomial of degree $n$.
Associated to $Q$ and $R$ is the 
resultant $\Res(R,Q)$. This is a polynomial of degree $n+2$ in the coefficients of $Q$ and $R$, which is defined over $\ZZ$.
Given any  fixed polynomial $R\in \ZZ[x]$ of degree  $n\geq 5$, it follows from work of  Schmidt \cite[Thm.~1]{Schmidt} that there are only finitely many  irreducible quadratic polynomials $Q\in \ZZ[x]$ such that $\Res(R,Q)=\pm 1$.
In this paper we shall concern ourselves  with the
opposite situation, and attempt to assess the 
 density of quadratic $Q\in \ZZ[x]$ 
for which 
$\Res(R,Q)=\pm 1$,  for some $R\in \ZZ[x]$ of odd degree. 
More specifically, we shall be interested in the size of the counting function
\begin{equation}\label{eq:NB'}
N(B)=\#\left\{(a,b,c)\in \ZZ^3: 
\begin{array}{l}
|a|,|b|,|c|\leq B\\
\Res(R,ax^2+bx+c)=\pm 1\\ 
\text{for some $R\in \ZZ[x]$ of odd degree}
\end{array}
\right\},
\end{equation}
as $B\to \infty$, with the aim of showing that $N(B)=o(B^3)$.

By using only mod $p$ information, we shall apply the large sieve to prove the following upper bound for $N(B)$.

\begin{theorem}\label{t:1}
We have 
$
N(B)\ll B^3/\sqrt{\log B}$.
\end{theorem}

It is also interesting to ask about the behaviour of the function
\begin{equation}\label{eq:NB}
N_n(B)=\#\left\{(a,b,c)\in \ZZ^3: 
\begin{array}{l}
|a|,|b|,|c|\leq B\\
\Res(R,ax^2+bx+c)=\pm 1\\ 
\text{for some $R\in \ZZ[x]$ of degree $n$}
\end{array}
\right\},
\end{equation}
for  fixed  $n\in \NN$.
Theorem \ref{t:1} upper bounds  $N_n(B)$ when $n$ is odd. 
Taking $R(x)=x^n$, it is easily checked that 
  $\Res(x^n,ax^2+bx+c)=c^n$. Hence we have 
  $
  N_n(B)\gg B^2,
  $ 
  coming from triples $(a,b,c)$ with $c=1$.
In private communication with the authors, Aaron Landesman has raised the following conjecture.

\begin{conjecture}[Landesman]\label{con}
Let $n\geq 2$. Then there exists $m\geq 0$ such that 
 $$N_n(B)\ll B^{2}(\log B)^m.$$
 \end{conjecture}

One might even expect that this upper bound holds with $m=0$ (at least when $n$ is odd). 
  We shall give some evidence towards Conjecture \ref{con}, by relating it to standard expectations around the typical size of $n$-torsion in the class group of imaginary quadratic fields.

Suppose that we are given a monic polynomial $R\in \ZZ[x]$ of odd degree, and a prime $p$. Then, 
as is well-known, a quadratic polynomial $Q\in \ZZ[x]$ will share 
a common root in $\overline\FF_p$ with $R$ if and only if 
 $p\mid \Res(R,Q)$.  It follows that the 
  quadratic polynomials $Q\in \ZZ[x]$ that share no 
 common root with $R$ in 
  $\overline\FF_p$ for any prime $p$, are precisely the 
  quadratic polynomials $Q\in \ZZ[x]$ for which $\Res(R,Q)=\pm 1$.
On appealing to Theorem~\ref{t:1}, we may conclude as follows.

\begin{corollary}
Let
$R\in \ZZ[x]$ be monic and of odd degree. Then almost all quadratic polynomials 
$Q\in \ZZ[x]$ share  a common root with $R$ in $\overline\FF_p$, for some prime $p$.
\end{corollary}

Elements of the 
class group of a quadratic number field $K/\QQ$ of discriminant $D_K$ are in bijection with equivalence classes 
of binary quadratic forms of discriminant $D_K$. This can be used to 
handle the average size of the class number $h(K)$, if one orders the quadratic number fields by discriminant. Thus it follows from Dirichlet's class number formula that 
\begin{equation}\label{eq:class}
\sum_{
\substack{K/\QQ \text{ quadratic}\\
|D_K| < X
}}  h(K)\ll   X^{\frac{3}{2}}.
\end{equation}
It is much more challenging  to assess the typical  size of 
$h_\text{odd}(K)$, which is defined to be the odd part of the class number, or even that of  
$h_n(K)$, for fixed $n$, which is defined to be size of the $n$-torsion subgroup of the  class group. 
In his thesis work \cite{aaron}, Landesman provided
a new geometric description of the $n$-torsion elements of the class groups of quadratic number fields, for a fixed $n$. Thus it follows from 
\cite[Thm.~1.1]{aaron} that a binary quadratic form $q$ corresponds to an odd order 
 element in the class group 
 if and only if there exists an odd integer $n$ 
and a degree $n$ polynomial whose resultant with $q$ is $\pm 1$. Combining this with the proof of 
Theorem \ref{t:1}, we will deduce the following result in Sections \ref{sec:disc1} and \ref{sec:disc2}.

\begin{corollary}\label{cor:1.2}
We have 
$$
\sum_{
\substack{K/\QQ \text{ quadratic}\\
|D_K|<X
}}  h_\text{odd}(K)\ll \frac{X^{\frac{3}{2}}}{\sqrt{\log X}}.
$$
\end{corollary}

It is also possible to arrive at this upper bound using Gauss' genus theory, which yields the lower  bound 
$h(K)=h_{\text{odd}}(K)h_{2^\infty}(K)\geq h_{\text{odd}}(K)2^{\omega(D_K)-1}$ for the class number. 
But then, on restricting to negative discriminants $D_K<-4$, it follows from 
Dirichlet's class number formula that 
\begin{equation}\label{eq:gerth}
\sum_{
\substack{K/\QQ \text{ quadratic}\\
-X<D_K<-4
}}  h_\text{odd}(K)\leq 
\frac{2}{\pi}
\sum_{
\substack{K/\QQ \text{ quadratic}\\
-X<D_K<-4
}}  \frac{\sqrt{|D_K|}L(1,\chi)}{2^{\omega(D_K)}},
\end{equation}
where 
 $L(s,\chi)$ is the Dirichlet $L$-function 
 with primitive real character $\chi(\cdot)=(\frac{D_K}{\cdot})$. 
 The sum on the right hand side can be estimated in a variety of ways, but recent work of Wilson \cite[Cor.~1.6]{wilson} yields an upper bound for 
\eqref{eq:gerth} that matches the one in Corollary~\ref{cor:1.2}.
For real quadratic fields $K/\QQ$ one can argue similarly, but this time taking the trivial lower bound $R_K\gg \log |D_K|$ for the regulator of $K$ in Dirichlet's class number formula.
In this way it is possible to obtain the improvement 
$$
\sum_{
\substack{K/\QQ \text{ quadratic}\\
0<D_K<X
}}  h_\text{odd}(K)
\ll 
\sum_{
\substack{K/\QQ \text{ quadratic}\\
0<D_K<X
}}  \frac{\sqrt{D_K}L(1,\chi)}{(\log D_K) 2^{\omega(D_K)}}
\ll\frac{X^{3/2}}{(\log X)^{\frac{3}{2}}}.
$$
 Gerth's  \cite{gerth1,gerth2} extension of the Cohen--Lenstra heuristics to the $2$-part of the class group implies that 
the upper bound in 
Corollary \ref{cor:1.2} is tight.

For a fixed positive integer $n$, 
there are many results in the literature 
that provide 
upper bounds for the average of $h_n(K)$
that save a power of $X$, rather than a power of $\log X$. 
After significant contributions by 
Heath-Brown--Pierce \cite{hb-p}, by 
Ellenberg--Pierce--Wood 
\cite{pierce2}, and 
by Frei--Widmer
\cite{frei},
the best general bound is
currently 
$$
\sum_{
\substack{K/\QQ \text{ quadratic}\\
|D_K|<X
}}  h_n(K)\ll X^{\frac{3}{2}
-\frac{1}{n+1}+\ve},
$$
for any $\ve>0$, and is 
due to  Koymans--Thorner \cite{koy}.
Note that any argument that yields a power-saving over the trivial  
bound for  $N_n(B)$
ought to yield a corresponding power-saving for the average of $h_n(K)$.
In Section \ref{s:3} we shall discuss some of the  issues involved when 
$n=3$. 

As a special case of the Cohen--Lenstra heuristics \cite{CL},
it is widely believed that 
\begin{equation}\label{eq:wide}
\sum_{
\substack{K/\QQ \text{ quadratic}\\
|D_K|<X
}}  h_n(K)\ll 
\begin{cases}
\hfil X&\text{if $n$ is odd},\\
\hfil X\log X&\text{if $n$ is even}.
\end{cases}
\end{equation}
It turns out that by using Landesman's correspondence, we can 
apply this bound to deduce an upper bound for the variant $N_n^-(B)$ of $N_n(B)$ in \eqref{eq:NB}, in which a restriction to quadratics with 
negative   discriminants is imposed. 
We shall prove the following result in Section \ref{sec:disc1}, the second part of which gives strong evidence towards Conjecture~\ref{con}.

\begin{theorem}\label{t:lama}
 \begin{enumerate}
\item 
We have $N_n^-(B)\ll_\ve 
B^{3-\frac{2}{n+1}+\ve}$,
for any $\ve>0$.
\item
Assume that \eqref{eq:wide} holds. Then $N_n^-(B)\ll B^2(\log B)^2$. 
\end{enumerate}
\end{theorem}

Unfortunately it seems difficult  to establish similar  upper bounds for the counting function $N_n(B)$, since for quadratics with positive discriminant we are required to 
estimate the number of  distinct integers with absolute value up to $x$ that are represented by a binary quadratic form with positive discriminant close to $x^2$, and such a bound would inevitably depend on the size of the fundamental unit.

\subsection*{Acknowledgements}
The  authors are very grateful to Aaron Landesman  
for drawing their attention to this problem and for sharing some useful remarks.
Special thanks are due to Fabian Gundlach for pointing out an error in an earlier version of this paper and for sharing helpful comments. Thanks are also due to 
Christopher Frei for useful feedback. 
While working on this paper
the first author was supported  by 
a FWF grant (DOI 10.55776/P36278).

\section{Resultants}

Let $Q(x)=ax^2+bx+c$ be a quadratic polynomial with $a,b,c\in \Z$, and let 
$R(x)=r_0x^n+r_1x^{n-1}+\dots+r_n$ be a polynomial of degree $n$ also defined over $\Z$.
The resultant of $R$ and $Q$ is defined to be 
the determinant of the $(n+2)\times (n+2)$ matrix $M$, where
$$
M=
\begin{pmatrix}
r_0 & 0& a& 0& \cdots  & 0 \\
r_1& r_0 &b& a&  \cdots& 0    \\
r_2& r_1&c& b& \cdots  & 0\\
r_3 &r_2&0& c& \cdots  & 0 \\
\vdots&\vdots& \vdots&\vdots & \ddots  & \vdots  \\
0& r_n&  0&  0& \cdots & c
\end{pmatrix}.
$$
In this section our main goal is to prove a formula for 
$\Res(R,Q)=\det M$ which is tailored to the problem of estimating $N(B)$ in \eqref{eq:NB}.
For any positive integer $k$, let $A_k$ be the $(k+2)\times k$ matrix
\begin{equation}\label{eq:Ak}
A_k=
\begin{pmatrix}
 a& 0& \cdots  & 0 \\
b& a&  \cdots& 0     \\
c& b& \cdots  & 0\\
0& c& \cdots  & a \\
 \vdots&\vdots & \ddots  & b  \\
  0&  0&\cdots & c
\end{pmatrix}.\end{equation}
In what follows, we shall denote by $B_{k,i,j}$  the $k\times k$ matrix formed by removing the $i$-th and $j$-th row from $A_{k}$, for any $1\leq i\leq k+2$ and $1\leq j\leq k$.
With the notation \eqref{eq:Ak}, if 
 $\mathbf{c}=(r_0,\dots, r_n)$, then the first two columns of $M$ 
are given by $(\mathbf{c},0)^T$ and $(0,\mathbf{c})^T$, 
 respectively, 
and the remaining $n$ columns are given by  $A_n$.

We begin by using elementary row operations to clear all but two of  the entries from the first two columns of the matrix whose determinant is being taken. 

\begin{lemma}\label{lem:res1}
Let $1\leq i<j\leq n+1$. Then there exist $X,Y\in\Q[r_0,\dots, r_n,a,b,c]$ depending on the choice of $i,j$ such that 
$$
\Res(R,Q)=
\det
\begin{pmatrix}
\vdots & \vdots & &&&& \\
X     & 0     & &&&&\\
0     & X     & &&&&\\
\vdots & \vdots &&&A_n&& \\
Y     & 0     &  &&&&\\ 
0     & Y     &     &&&&\\
\vdots & \vdots & &&&&
\end{pmatrix},
$$
where $X$ appears in the $(i,1)$-entry and $(i+1,2)$-entry of the matrix,  and $Y$ appears in the $(j,1)$-entry and $(j+1,2)$-entry.
\end{lemma}

\begin{proof}
Let $i,j$ be fixed, with  $1\leq i<j\leq n+1$.
Let $\mathbf{c}$ be the column vector $(r_0,\dots, r_n)^T$
and 
let $\mathbf{c}_{i,j}$ denote the column vector formed by removing the $i$-th and  $j$-th entries from $\mathbf{c}$. 
 Recalling the definition \eqref{eq:Ak} of $A_k$, and the subsequent definition of 
 $B_{k,i,j}$, 
we claim that we  can find a linear combination over $\Q$ of the  columns of $A_n$ to 
clear any $n-1$ of the $n+1$ non-trivial entries of the first column. 
Indeed, we begin by solving 
\[B_{n-1,i,j}\boldsymbol{\lambda}=\mathbf{c}_{i,j},\]
for a vector $\boldsymbol{\lambda}=(\lambda_1,\dots,\lambda_{n-1})^T\in\Q^{n-1}$.
We then proceed to perform  the column operation which subtracts $A_n
(\lambda_1,\dots,\lambda_{n-1},0)^T$
from the first column of $M$, or 
$A_n
(0,\lambda_1,\dots,\lambda_{n-1})^T$ from the second
	column of $M$.
The $n+2$ rows of $A_n$ are given by 
\begin{align*}
\mathbf{a}_1&=(a,0,\dots,0),  &\mathbf{a}_2=(b,a,\dots,0),\\
\mathbf{a}_{n+1}&=(0,\dots ,0,c,b),
&\mathbf{a}_{n+2}=(0,\dots ,0,c), 
\end{align*}
and 
$$
\mathbf{a}_{k}=(0,\dots,0,c,b,a,0,\dots,0), 
$$
for $3\leq k\leq n$, where $c,b,a$ occur in the $k-2,k-1,k$-th entry, respectively.
In conclusion, the procedure we have described produces $X,Y\in\Q[r_0,\dots, r_n,a,b,c]$, given by
\begin{align*}
X&=r_{i-1}-\mathbf{a}_{i}\begin{pmatrix}\boldsymbol{\lambda}\\ 0\end{pmatrix}
=r_{i-1}-c\lambda_{i-2}\mathbf{1}_{i\geq 3}-b\lambda_{i-1}\mathbf{1}_{i\geq 2}-a\lambda_{i}\mathbf{1}_{i\leq n-1},\\
Y&=r_{j-1}-\mathbf{a}_{j}\begin{pmatrix}\boldsymbol{\lambda}\\ 0\end{pmatrix}=r_{j-1}-c\lambda_{j-2}\mathbf{1}_{j\geq 3}-b\lambda_{j-1}\mathbf{1}_{j\leq n}
-a\lambda_j\mathbf{1}_{j\leq n-1},
\end{align*}
 such that the statement of the lemma holds. 
\end{proof}

Using this expression for the resultant  we are now in a position to compute the determinant.

\begin{lemma}\label{lem:res2}
Let $1\leq i<j\leq n+1$.
Then there exist $X,Y\in\Q[r_0,\dots, r_n,a,b,c]$ depending on the choice of $i,j$ such that 
\[
\Res(R,Q)
=a^{i-1}c^{n-j+1}\left(c^{j-i}X^2+(-1)^{j-i}(d_{j-i}-acd_{j-i-2})XY+a^{j-i}Y^2\right),
\]
where 
$d_{-1}=0$,  $d_0=1$, and $d_k$ is the $k\times k$ determinant 
\[d_k=\det
\begin{pmatrix}
b & a & 0& \cdots  & &  & \\
c & b& a&  \cdots & &  0&\\
0 & c& b& \cdots &  &  & \\
\vdots & \vdots &\vdots & \ddots  &\vdots &\vdots &\vdots \\ 
   &    &  & \dots& b& a&  0\\
   & 0  &   &\dots& c& b&  a\\
   &    &  &\dots & 0& c&  b
\end{pmatrix},
\]
for $k\in \NN$.
\end{lemma}

\begin{proof}
Recall the definition \eqref{eq:Ak} of $A_k$, 
and the subsequent definition of $B_{k,i,j}$.
(In particular $d_k=\det B_{k,1,k+2}$ in this notation, for $k\geq 1$.)
 Let $i,j$ be fixed such that  $1\leq i<j\leq n+1$. Then 
it follows from Lemma \ref{lem:res1} that 
\[
\Res(R,Q)=(-1)^{i+1}X\det M_{i,1}+(-1)^{j+1}Y\det M_{j,1},
\]
where $M_{k,1}$ is the $(n+1)\times(n+1)$ matrix by removing the $k$-th row and the first column.
Now the matrix $M_{i,1}$ has an $X$ in its $(i,1)$-entry and a $Y$ in its $(j,1)$-entry, so
\[\det M_{i,1}=(-1)^{i+1}X\det B_{n,i,i+1}+(-1)^{j+1}Y\det B_{n,i,j+1}.\]
Similarly, 
the matrix $M_{j,1}$ has an $X$ in its $(i+1,1)$-entry if $j>i+1$, and a $Y$ in its $(j,1)$-entry, so
\[\det M_{j,1}=(-1)^{i+2}X\mathbf{1}_{j> i+1}\det B_{n,i+1,j}+(-1)^{j+1} Y\det B_{n,j,j+1}.\]
Putting everything together, we obtain
\begin{equation}\label{eq:Res1}
\begin{split}
\Res(R,Q)
=~&
X^2\det B_{n,i,i+1}
+Y^2 \det B_{n,j,j+1}\\
&+(-1)^{j-i}\left(\det B_{n,i,j+1}-\mathbf{1}_{j> i+1}\det B_{n,i+1,j}\right)XY.
\end{split}
\end{equation}

It remains to compute $\det B_{n,i,j+1}$ for 
 integers $1\leq i\leq j\leq n+1$. 
If $i>1$,  Laplace expansion along the first row of $B_{n,i,j+1}$ gives
\[ \det B_{n,i,j+1}=a\det B_{n-1,i-1,j}.\]
If $j\leq n$, Laplace expansion along the final row of $B_{n,i,j+1}$ gives
\[ \det B_{n,i,j+1}=c\det B_{n-1,i,j+1}.\]
Applying these relations recursively, we easily obtain
\[ \det B_{n,i,j+1}
=a^{i-1}\det B_{n-i+1,1,j-i+2}
=a^{i-1}c^{n-j+1}
\det B_{j-i,1,j-i+2}.
\]
Recalling the definition of $d_k$ from the statement, 
we see that 
\[\det B_{n,i,j+1}=a^{i-1}c^{n-j+1}d_{j-i},
\]
which  easily leads to the statement of the lemma, on insertion into  \eqref{eq:Res1}.
\end{proof}

Each $d_k$ is a homogeneous polynomial of degree $k$ in $a,b,c$ and can be evaluated as follows.

\begin{lemma}\label{lem:res3}
Assume that $b^2-4ac\neq 0$. 
Then for any  $k\in \NN$, we have
\[
d_k=\frac{1}{\sqrt{b^2-4ac}}\left(\left(\frac{b+\sqrt{b^2-4ac}}{2}\right)^{k+1}-\left(\frac{b-\sqrt{b^2-4ac}}{2}\right)^{k+1}\right).
\]
Furthermore, for $k\geq 2$, we have
\[d_k-acd_{k-2}=\left(\frac{b+\sqrt{b^2-4ac}}{2}\right)^{k}+\left(\frac{b-\sqrt{b^2-4ac}}{2}\right)^{k}.\]
\end{lemma}
\begin{proof}
This follows easily from the recurrence relation $d_k=bd_{k-1}-acd_{k-2}$, and the initial conditions $d_{1}=b$ and  $d_2=b^2-ac$.
\end{proof}

\section{The cubic case: a worked example}\label{s:3}

In this section we 
discuss the counting function $N_3(B)$ in \eqref{eq:NB},
which we shall use to 
illustrate the  calculations in the preceding section.
When $R\in \ZZ[x]$ is cubic we have
\begin{align*}
\Res(R,Q)
&=
\det
\begin{pmatrix}
r_0 & 0 & a& 0&  0    \\
r_1 & r_0& b& a& 0    \\
r_2 & r_1& c& b& a   \\
r_3 & r_2& 0& c& b  \\
0 & r_3& 0& 0&  c    \\
\end{pmatrix}\\
&=a^3 r_3^2
-a^2 b r_2 r_3-2 a^2 c r_1 r_3 + a^2 c r_2^2+
a b^2 r_1 r_3 + 3 a b c r_0 r_3 \\
&\quad - a b c r_1 r_2 - 2 a c^2 r_0 r_2 +
a c^2 r_1^2
-b^3 r_0 r_3 + b^2 c r_0 r_2-b c^2 r_0 r_1 + c^3 r_0^2.
\end{align*}
The points $(a,b,c,r_0,r_1,r_2,r_3)$ on the hypersurface
$\Res(R,Q)=-1$ are in bijection with the 
 points $(-a,-b,-c,r_0,r_1,r_2,r_3)$ on the hypersurface 
$\Res(R,Q)=1$, and so we may focus our attention on the latter equation. 
This 
 defines a quintic hypersurface in $\mathbb{A}^7=\mathbb{A}_{a,b,c}^3\times \mathbb{A}_{r_0,\dots,r_3}^4$ and has the structure of a quadric bundle over $\mathbb{A}_{a,b,c}^3$, the fibres of which are quadrics in  $\mathbb{A}_{r_0,\dots,r_3}^4$. In this way we can interpret $N_3(B)$ as counting integral points of height at most $B$ in $\mathbb{A}_{a,b,c}^3$ for which the fibre has 
 an integral point. However, it turns out that these quadrics are all degenerate and we are led to study 
the existence of integral points in families of 
 irreducible affine conics.

We can get explicit equations for the relevant conics by revisiting the calculations in 
Lemma \ref{lem:res1}.
The set 
$\{(1,2),(1,3),(1,4),(2,3),(2,4),(3,4)\}$  comprises the set of allowable indices $(i,j)$,
so that we always have  $j-i\in\{1,2,3\}$.
If $(i,j)=(1,2)$, then
\[
\Res(R,Q)=
\det\begin{pmatrix}
X_{1,2} & 0 & a& 0&  0    \\
Y_{1,2} & X_{1,2}& b& a& 0    \\
0 & Y_{1,2}& c& b& a   \\
0 & 0& 0& c& b  \\
0 & 0& 0& 0&  c    
\end{pmatrix}=c^2 (c X_{1,2}^2- b X_{1,2} Y_{1,2} +a Y_{1,2}^2  ),\]
where 
\[\begin{pmatrix}
X_{1,2}\\
Y_{1,2}\end{pmatrix}
=\begin{pmatrix}
r_0\\
r_1\end{pmatrix}
-\begin{pmatrix}
a & 0     \\
b & a 
\end{pmatrix}\begin{pmatrix}
\lambda_1    \\
\lambda_2   
\end{pmatrix}\quad\text{ and }\quad\begin{pmatrix}
c & b    \\
0 & c
\end{pmatrix}\begin{pmatrix}
\lambda_1    \\
\lambda_2   \\  
\end{pmatrix}=\begin{pmatrix}
r_2    \\
r_3   
\end{pmatrix}.\]
If $(i,j)=(1,3)$, then $\Res(R,Q)$ is 
\[\det\begin{pmatrix}
X_{1,3} & 0 & a& 0&  0    \\
0 & X_{1,3}& b& a& 0    \\
Y_{1,3} & 0& c& b& a   \\
0 & Y_{1,3}& 0& c& b  \\
0 & 0& 0& 0&  c    \\
\end{pmatrix}=c (c^2 X_{1,3}^2+(b^2- 2 a c) X_{1,3} Y_{1,3} +a^2 Y_{1,3}^2),\]
where 
\[\begin{pmatrix}
X_{1,3}\\
Y_{1,3}\end{pmatrix}
=\begin{pmatrix}
r_0\\
r_1\end{pmatrix}
-\begin{pmatrix}
a & 0   \\
c & b
\end{pmatrix}\begin{pmatrix}
\lambda_1    \\
\lambda_2  \end{pmatrix}\quad\text{ and }\quad\begin{pmatrix}
b & a     \\
0 & c 
\end{pmatrix}\begin{pmatrix}
\lambda_1    \\
\lambda_2  
\end{pmatrix}=\begin{pmatrix}
r_2    \\
r_3 
\end{pmatrix}.\]
If $(i,j)=(1,4)$, then $\Res(R,Q)$ is
\[\det\begin{pmatrix}
X_{1,4} & 0 & a& 0&  0    \\
0 & X_{1,4}& b& a& 0    \\
0 & 0& c& b& a   \\
Y_{1,4} & 0& 0& c& b  \\
0 & Y_{1,4}& 0& 0&  c    \\
\end{pmatrix}=c^3 X_{1,4}^2 - (b^3 - 3 a b c) X_{1,4} Y_{1,4} + a^3 Y_{1,4}^2,\]
where 
\[\begin{pmatrix}
X_{1,4}\\
Y_{1,4}\end{pmatrix}
=\begin{pmatrix}
r_0\\
r_1\end{pmatrix}
-\begin{pmatrix}
a & 0    \\
0 & c
\end{pmatrix}\begin{pmatrix}
\lambda_1    \\
\lambda_2   
\end{pmatrix}\quad\text{ and }\quad\begin{pmatrix}
b & a     \\
c & b
\end{pmatrix}\begin{pmatrix}
\lambda_1    \\
\lambda_2   
\end{pmatrix}=\begin{pmatrix}
r_2    \\
r_3  
\end{pmatrix}.\]
These calculations are all consistent with  Lemmas \ref{lem:res2} and \ref{lem:res3}.

When one makes  all these substitutions for $X_{i,j},Y_{i,j}$ one see that 
$N_3(B)$ is bounded by the number of  $(a,b,c)\in \ZZ^3$ for which 
$|a|,|b|,|c|\leq B$ and the  equations
\begin{align*}
cx_{2}^2-bx_{2}y_{2}+ay_{2}^2&=c^2,\\
c^2 x_3^2+(b^2- 2 a c) x_{3} y_3 +a^2 y_{3}^2&= b^2c, \\
c^3x_{4}^2-(b^3-3abc)x_{4}y_{4}+a^3y_{4}^2&= (b^2-ac)^2,
\end{align*}
each admit a solution $(x_j,y_j)\in \ZZ^2$, for $2\leq j\leq 4$. Testing for solubility of generalised Pell equations presents a formidable challenge, as exemplified in recent work \cite{chan, pell} on 
 the solubility of the negative Pell equation.

\section{The large sieve}

We 
now  turn to the task of estimating the counting function $N(B)$ that was defined in 
\eqref{eq:NB'}. 
Note that there are $O(B^2)$ choices  of $(a,b,c)$ with $abc= 0$.
Moreover, the resultant $\Res(R,ax^2+bx+c)$ is homogeneous of odd degree 
in the coefficients $a,b,c$. Hence, on replacing $(a,b,c)$ by $(-a,-b,-c)$, we have 
$$
N(B)\leq 
2\#\left\{(a,b,c)\in \ZZ^3: 
\begin{array}{l}
0<|a|,|b|,|c|\leq B\\
\Res(R,ax^2+bx+c)=1\\ 
\text{for some $R\in \ZZ[x]$ of odd degree}
\end{array}
\right\}
+O(B^2).
$$
We now suppose that 
$\Res(R,ax^2+bx+c)=1$\ 
for some $R\in \ZZ[x]$ of odd degree $n\geq 3$.
Then, for any choice of indices $(i,j)$  with $1\leq i<j\leq n+1$, it follows from Lemma \ref{lem:res2} that 
there exist $X,Y\in \QQ$ such that 
$$
a^{i-1}c^{n-j+1}\left(c^{j-i}X^2+(-1)^{j-i}(d_{j-i}-acd_{j-i-2})XY+a^{j-i}Y^2\right)= 1, 
$$
where formulae for $d_{j-i}$ and $d_{j-i-2}$ are given by Lemma \ref{lem:res3} in 
 terms of $a,b,c$.
 Letting $k=j-i$, we may conclude that the conic $C_{i,j}\subset \PP^2$ 
has a rational point, where $C_{i,j}$ is given by 
$$
c^{k}X^2+(-1)^{k}(d_{k}-acd_{k-2})XY+a^{k}Y^2= a^{i-1}c^{n-j+1}Z^2.
$$
Unfortunately, there is no advantage to be gained from 
examining the existence of rational points on more than  one conic.
 Taking $(i,j)=(1,2)$ and making an obvious change of variables, we see that $C_{1,2}(\QQ)\neq \emptyset$ if and only if 
the conic 
$$
C: \quad X^2-(b^2-4ac)Y^2=cZ^2
$$
has a $\QQ$-point.  It is worth highlighting that this is the key step at which our assumption on the parity of the degree of $R$  enters the argument.  We conclude that 
\begin{equation}\label{eq:comic}
N(B)\leq 2N_+(B)+O(B^2),
\end{equation}
with 
$$
N_+(B)=
\#\left\{(a,b,c)\in \ZZ^3: 
0<|a|,|b|,|c|\leq B,~
C(\QQ)\neq \emptyset 
\right\}.
$$
If $C(\QQ)\neq \emptyset$  then necessarily 
$(\frac{c}{p})\neq -1$
for any prime $p$ such that $v_p(b^2-4ac)=1$. 
Thus
\begin{equation}\label{eq:book}
N_+(B)\leq 
T(B,B,B),
\end{equation}
where 
if $B_1,B_2,B_3\geq 1$ we put 
$$
T(B_1,B_2,B_3)=
\#\left\{(a,b,c)\in \Omega: 
|a|\leq B_1,~|b|\leq B_2,~|c|\leq B_3
\right\},
$$
 where
$$
\Omega=
\left\{(a,b,c)\in \ZZ^3:
v_p(b^2-4ac)=1\Rightarrow \left(\frac{c}{p}\right)\neq -1
\text{ for all primes $p$}\right\}.
$$

We shall approach the problem of estimating the right hand side via the
large sieve, 
in the form 
\cite[Lemma~6.3]{large}.  
We first record an upper bound for the size of the reduction modulo $p^2$ of $\Omega$.

\begin{lemma}
Then $\#(\Omega \bmod{p^2})\leq (1-\omega_p)p^6$, where
$\omega_p=\frac{1}{2p}+O(\frac{1}{p^2}).$
\end{lemma}

\begin{proof}
We may clearly assume that $p>2$. We have $\#(\Omega \bmod{p^2})=p^6-S_p$, where
$S_p$ is the number of $(a,b,c)\in (\ZZ/p^2\ZZ)^3$ for which $v_p(b^2-4ac)=1$ and 
$(\frac{c}{p})= -1$.
Clearly
\begin{align*}
S_p\geq~& 
p^3\#\left\{ (a,b,c)\in \FF_p^3: b^2-4ac=0, ~\left(\frac{c}{p}\right)= -1\right\}\\
&-\#\left\{
(a,b,c)\in (\ZZ/p^2\ZZ)^3: b^2-4ac\equiv 0\bmod{p^2}\right\}.
\end{align*}
The  first summand is  $\frac{1}{2}p^4(p-1)$
and the second summand is $O(p^4)$.
Hence we deduce that
$S_p\geq \frac{1}{2}p^5+O(p^4)$, which concludes the proof.
\end{proof}

It now follows from 
\cite[Lemma~6.3]{large} that 
\begin{align*}
T(B_1,B_2,B_3)
\ll \frac{(B_1+Q^4)(B_2+Q^4)(B_3+Q^4)}{L(Q)},
\end{align*}
for any $Q\geq 1$,
where
$$
L(Q)=\sum_{1\leq q\leq Q}\mu^2(q)\prod_{p\mid q} \frac{\omega_p}{1-\omega_p}\geq 
\sum_{1\leq q\leq Q}\frac{\mu^2(q)}{\tau(q)q}\prod_{p\mid q} \left(1-\frac{1}{p}\right)^C
$$
for an appropriate absolute constant $C\geq 1$.
A straightforward application of 
 \cite[Thm.~A.5]{FrIwOdC} now yields 
$
L(Q)\gg \sqrt{\log Q},
$
from which it follows that 
\begin{equation}\label{eq:Tve}
T(B_1,B_2,B_3)\ll \frac{B_1B_2B_3}{\sqrt{\log \min \{B_1,B_2,B_3\}}},
\end{equation}
on taking $Q=\min \{B_1,B_2,B_3\}^{\frac{1}{4}}$.
Theorem \ref{t:1} now follows on 
taking $B_1=B_2=B_3=B$ and 
inserting this into \eqref{eq:comic} and \eqref{eq:book}.

\section{Ordering by discriminant}

\subsection{Negative discriminant}\label{sec:disc1}

Let 
$$
S^-(X)=
\sum_{
\substack{K/\QQ \text{ quadratic}\\
-X<d<0
}}  h_\text{odd}(K),
$$
where $d=D_K$ is the discriminant of $K$.
In this section we shall compare 
$S^-(X)$ with variants of the counting function 
\begin{equation}\label{eq:NB-}
N^-(B)=\#\left\{(a,b,c)\in \ZZ^3: 
\begin{array}{l}
|a|,|b|,|c|\leq B, ~b^2<4ac\\
\Res(R,ax^2+bx+c)=\pm 1\\ 
\text{for some $R\in \ZZ[x]$ of odd degree}
\end{array}
\right\},
\end{equation}
in which a restriction to negative discriminants is imposed in the counting function 
\eqref{eq:NB'}.
We begin by proving the following result, which takes care of the negative discriminants in Corollary 
\ref{cor:1.2}.

\begin{lemma}\label{lem:plant}
We have 
$$
S^-(X)\ll \frac{X^{\frac{3}{2}}}{\sqrt{\log X}}. 
$$
\end{lemma}

\begin{proof}
As described by Buell \cite[Prop.~6.20]{Buell}, 
there is an isomorphism between the set of equivalence classes of primitive binary quadratic forms of discriminant $d$ and  the narrow ideal class group of $K$. A quadratic form $q(x,y)=ax^2+bxy+cy^2$ with discriminant $d<0$ is \emph{reduced} (in the sense of \cite[Eq.~(2.1)]{Buell}) if 
$$
|b|\leq a\leq c.
$$
Any binary quadratic form with 
discriminant  $d$ is equivalent to a reduced quadratic form of 
discriminant  $d$ by  \cite[Thm.~2.3]{Buell}.
In the light of the work by Landesman \cite[Thm.~1.1]{aaron}, the sum
$S^-(X)$ is bounded by the number of 
reduced primitive positive definite quadratic forms $q$ of discriminant  $d\in (-X,0)$ for which 
there exists 
polynomial $R\in \ZZ[x]$ of odd degree whose   resultant with $q(x,1)$ is $\pm 1$. 
If $|d|=4ac-b^2\geq 3a^2$, so that 
$a\ll \sqrt{d}\ll \sqrt{X}$.
If $|b|\leq X^{\frac{1}{2}-\frac{1}{10}}$, then
$|d|\leq X$ implies that $ac\ll X$.  
Dropping the resultant condition, we see that the overall contribution to $S^-(X)$ from 
such binary quadratic forms 
is  $O(X^{\frac{3}{2}-\frac{1}{10}}\log X)$, which is satisfactory. 
Therefore we can proceed under the assumption that $|b|>X^{\frac{1}{2}-\frac{1}{10}}$.
Since $b^2\leq a^2\ll |d|$, we have $ac\ll|d|+b^2\ll X$.
We therefore  concentrate on the set of $(a,b,c)\in \Z_{\geq 0}\times \Z\times \Z_{\geq 0}$ which are constrained by the inequalties
\[X^{\frac{1}{2}-\frac{1}{10}}<|b|\leq a\ll  \sqrt{X}\quad \text{ and }\quad a\leq c\ll \frac{1}{a}X.\]
It will be convenient to sort the contribution into dyadic intervals, with 
\begin{equation}\label{eq:fox}
A\leq a<2A,\quad B\leq |b|<2B,\quad C\leq c<2C,
\end{equation}
where $A,B,C$ run over powers of $2$ and satisfy
\begin{equation}\label{eq:trot}
X^{\frac{1}{2}-\frac{1}{10}}\ll B\ll A\ll\sqrt{X}\quad \text{ and }\quad A\ll C\ll \frac{1}{A}X.
\end{equation}
Let $S^-(A,B,C)$ denote the overall contribution to $S^-(X)$ from the binary quadratic forms 
$ax^2+bxy+cy^2$ whose coefficients are constrained by \eqref{eq:fox}.
Then it follows from \eqref{eq:Tve}
that 
$$
S^-(A,B,C)\ll \frac{ABC}{\sqrt{\log B}}\ll 
\frac{ABC}{\sqrt{\log X}},
$$
since $C\gg A\gg B\gg X^{\frac{1}{2}-\frac{1}{10}}$. 
It remains to  sum this  estimate for 
 $A,B,C$ running over powers of $2$ that are constrained by the 
inequalities in \eqref{eq:trot}. This yields
\begin{align*}
S^-(X)\leq \sum_{A,B,C} S^-(A,B,C)
&\ll \sum_{A,C} 
\frac{A^2C}{\sqrt{\log X}}\\
&\ll \frac{X}{\sqrt{\log X}}
\sum_{A\ll \sqrt{X}} 
A\\
&\ll \frac{X^{\frac{3}{2}}}{\sqrt{\log X}},
\end{align*}
which thereby completes the proof of the lemma.
\end{proof}

The latter result shows how an upper bound for a variant of the counting function \eqref{eq:NB-},
in which  lop-sided box sizes are allowed, can be used to provide an upper bound for $S^-(X)$. 
Fixing a choice of integer $n\in \NN$, 
our next result reverses the process and shows how 
$N_n^-(B)$ can be bounded in terms of the sum 
$$
S_n^-(X)=
\sum_{
\substack{K/\QQ \text{ quadratic}\\
-X<d<0
}}  h_n(K).
$$

\begin{lemma}
Let $n\in \NN$. Then 
$$
N_n^-(B)\ll 
B^2 \sum_{j=0}^{2\log_2 B+O(1)}
 \frac{S_n^-(2^j)}{2^j}.
$$
\end{lemma}

\begin{proof}
Note that we can trivially impose that $\gcd(a,b,c)=1$, because we always have $\gcd(a,b,c)^n\mid \Res(R,ax^2+bx+c)$ for any $R\in\Z[x]$ with degree $n$. 
The discriminant of $q(x,y)=ax^2+bxy+cy^2$ can be written as $Dr^2$, where $D$ is a fundamental discriminant and $r\in\Z$ is non-zero. 
In particular $|D|\leq |D|r^2\ll B^2$, since $\max\{|a|,|b|,|c|\}\leq B$.

By \cite[Prop.~7.1]{Buell}, there exists a form $f$ of discriminant $D$, and $\alpha,\beta,\gamma,\delta \in\Z$, such that $f(\alpha x+\beta y,\gamma x+\delta y)=q(x,y)$ and $\alpha\delta-\beta\gamma=r$.
The condition $|a|,|c|\leq B$ implies that
\begin{equation}\label{eq:pair}
|f(\alpha ,\gamma )|=|a|\leq B, \text{ and }|f(\beta ,\delta )|=|c|\leq B.
\end{equation}
Without loss of generality, we can take $f(x,y)=a_0x^2+b_0xy+c_0y^2$ to be a reduced form, so that in particular $0\leq a_0\ll \sqrt{|D|}\ll B$. Since $D=b_0^2-4a_0c_0<0$, we must have $a_0c_0\neq0$. Thus, on writing 
$$
f(x,y)=\frac{1}{4a_0}\left((2a_0x+b_0y)^2+|D|y^2\right),
$$ 
we see that $|f(x ,y)|\leq B$ implies that $|y|\leq\sqrt{ \frac{4a_0B}{|D|}}$ and $|x+\frac{b_0}{2a_0}y|\leq \sqrt{\frac{B}{a_0}}$. Therefore
\[\#\{(x,y)\in\Z^2: |f(x,y)|\leq B,\ y\neq 0\}\ll \sqrt{ \frac{4a_0B}{|D|}}\left( \sqrt{\frac{B}{a_0}}+1\right)\ll \frac{B}{\sqrt{|D|}}.
\]
By symmetry, it follows that 
\[\#\{(x,y)\in\Z^2: |f(x,y)|\leq B,\ (x,y)\neq (0,0)\}\ll \frac{B}{\sqrt{|D|}}.\]
Given $f$, we may conclude that there are $O(B^2/|D|)$ 
choices of  $(\alpha ,\gamma )$ and $(\beta ,\delta )$ that satisfy~\eqref{eq:pair}.

The condition that $\Res(R,q)=\pm 1$ for some $R\in\Z[x]$ of degree $n$ implies that 
$\Res(R(dx-b,-cx+a),f(x))=\Res(R(x),f(ax+b,cx+d))=\pm 1$.
Appealing to \cite[Thm.~1.1]{aaron}, given any fundamental discriminant $D<1$, we have the upper bound
\[
\#\left\{(a,b,c)\in \ZZ^3: 
\begin{array}{l}
|a|,|b|,|c|\leq B\\
b^2-4ac=Dr^2\text{ for some }r\in\Z\\
\Res(R,ax^2+bx+c)=\pm 1\\ 
\text{for some $R\in \ZZ[x]$ of degree $n$}
\end{array}
\right\}
\ll \frac{B^{2}}{|D|}h_n(\Q(\sqrt{D})).
\]
Finally we sum over all $D$ in dyadic intervals up to $D\ll B^2$ 
to conclude that 
\begin{align*}
N_n^-(B)
&\ll B^2 \sum_{\substack{D<0\\ |D|\ll B^2}} \frac{h_n(\Q(\sqrt{D}))}{|D|}
\ll B^2 \sum_{X=2^j \ll B^2} \frac{S_n^-(X)}{X},
\end{align*}
as claimed in the lemma.
\end{proof}

Using this result we may easily complete the proof of 
Theorem \ref{t:lama}. The second part is a direct consequence of \eqref{eq:wide}, since 
 $j=O(\log B)$. For the first part we apply the work 
of Koymans and Thorner
\cite{koy}, which gives 
$$
N_n^-(B)\ll_\ve 
B^2 \sum_{2^j \ll B^2} 
(2^j)^{\frac{1}{2}-\frac{1}{n+1}+\ve}\ll B^{3-\frac{2}{n+1}+2\ve},
$$
for any $\ve>0$.

\subsection{Positive discriminant}\label{sec:disc2}

Let 
$$
S^+(X)=
\sum_{
\substack{K/\QQ \text{ quadratic}\\
0<d<X
}}  h_\text{odd}(K),
$$
where $d=D_K$ is the discriminant of $K$.
In this section we shall prove the following result, which completes the proof of Corollary \ref{cor:1.2}, once combined with Lemma \ref{lem:plant}.

\begin{lemma}\label{lem:fish}
We have 
$$
S^+(X)\ll \frac{X^{\frac{3}{2}}}{\sqrt{\log X}}.
$$
\end{lemma}

We shall prove this result by  relating $S^+(X)$  to the counting function
$$
N^+(B)=\#\left\{(a,b,c)\in \ZZ^3: 
\begin{array}{l}
|a|,|b|,|c|\leq B, ~b^2>4ac\\
\Res(R,ax^2+bx+c)=\pm 1\\ 
\text{for some $R\in \ZZ[x]$ of odd degree}
\end{array}
\right\},
$$
in which a restriction to positive discriminants is imposed in
\eqref{eq:NB'}.  Thus, since $N^+(B)\leq N(B)$, 
 Lemma \ref{lem:fish} is a direct consequence of 
applying Theorem~\ref{t:1} in 
 the following result.

\begin{lemma}
We have 
$
S^+(X)\leq N^+( \sqrt{X}).
$
\end{lemma}

\begin{proof}
We adapt the arguments in the previous section to deal with quadratic number fields of positive  discriminant.  A quadratic form $q(x,y)=ax^2+bxy+cy^2$ with discriminant $d>0$ is \emph{reduced} (in the sense of \cite[Eq.~(3.1)]{Buell}) if 
 \[0<b<\sqrt{d},\quad \sqrt{d}-b<2|a|<\sqrt{d}+b.
 \]
Moreover,  a reduced form with discriminant $d>0$ further satisfies
 \[\sqrt{d}-b<2|c|<\sqrt{d}+b,\]
 by 
  \cite[Prop.~3.1]{Buell}.  Thus 
$$ |a|< \sqrt{d} ,\quad 0<b<\sqrt{d},\quad |c|<\sqrt{d},
$$
for any reduced binary quadratic form $q(x,y)$ with discriminant $d>0$.

Any binary quadratic form with 
discriminant  $d$ is equivalent to a reduced quadratic form of 
discriminant  $d$ by  \cite[Prop.~3.3]{Buell}. It follows from the work of Landesman \cite[Thm.~1.1]{aaron} that the sum
$S^+(X)$ is bounded by the number of 
reduced primitive positive definite quadratic forms $q$ of discriminant  $d\in (0,X)$, for which 
there exists 
an odd degree polynomial $R\in \ZZ[x]$ whose   resultant with $q(x,1)$ is $\pm 1$. 
Hence, in order to bound  $S^+(X)$, we can restrict to triples $(a,b,c)\in\Z$ in the range
$\max\{|a|,|b|,|c|\}\leq \sqrt{X}$
and the statement easily follows. 
\end{proof}

\end{document}